\documentclass[12pt]{amsart}
\usepackage{amssymb}
\usepackage{epsfig}
\usepackage{amscd}

\newcommand{\bc}{{\mathbb C}}

\newcommand{\br}{{\mathbb R}}

\newcommand{\ra}{{\rightarrow}}
\newtheorem{thm}{Theorem}
\newtheorem{Prop}{Proposition}
\newtheorem{lemma}{Lemma}
\newtheorem{co}{Corollary}
\newtheorem{Def}{Definition}

\newenvironment{pf}{\begin{trivlist}\item[]{\bf Proof:\ }}
{\hfill\rule{.08in}{.08in}\end{trivlist}}

\let \cal \mathcal
\begin{document}
\title[Spectral finiteness]
{Isospectral finiteness on  convex cocompact hyperbolic 3-manifolds}
\author{Gilles Courtois and Inkang Kim }
\date{}
\maketitle

\begin{abstract}
In this paper we show that a given set of lengths of closed
geodesics, there are only finitely many convex cocompact hyperbolic
3-manifolds with that
specified length spectrum, homotopy equivalent to a given 3-manifold without a handlebody factor, up to orientation preserving isometries.
\end{abstract}
\footnotetext[3]{The second author gratefully acknowledges the
partial support of grant  (NRF-2014R1A2A2A01005574), CNRS, Ecole
Polytechnique and  l'Institut de Math\'ematiques de Jussieu during his visit.}
\section{Introduction}
One of the main theme of Riemannian geometry is to find a criterion
for two manifolds to be isometric. Along the development in this
direction, the study of closed geodesics proved to be very useful in
many cases. The {\bf length spectrum} $\Lambda(M)$ of a Riemannian
manifold $M$ is the set of lengths of closed geodesics {\bf with
multiplicity}. If we consider the length spectrum without
multiplicity, we will say so explicitly. There exists a famous
conjecture called {\sl conjugacy problem}. If two compact negatively
curved manifolds have the {\sl same marked length spectrum}, then it
is conjectured that they are isometric. There is very little
progress on this conjecture except when one manifold is locally
symmetric \cite{BCG} or both are surfaces \cite{Otal, Croke}.

There is an anologous conjecture. It states that the set of
negatively curved metrics on a fixed compact manifold with a fixed
set of lengths of closed geodesics, forms a compact (or finite) set
in the space of Riemannian metrics. One can rephrase this as an {\sl
unmarked length rigidity}. There is also very little progress in
this direction. There is a close relationship between length spectrum
and Laplace spectrum. For negatively curved compact manifolds, the
latter determines the former, see \cite{V}. For surface case, two
notions are equivalent by Selberg trace formula.  McKean in the
70's showed \cite{Mc} that there exists only finite number of hyperbolic
metrics on a surface with a given {\sl Laplace spectrum}. Later
Osgood, Phillips and Sarnak \cite{Os} showed the compactness of
Laplacian isospectral metrics on a closed surface. Much later,
Brooks, Perry and Petersen \cite{Brooks} showed the same result for
closed 3-manifolds near a metric of constant curvature. More
recently, Croke and Sharafutdinov \cite{Croke} showed the non-existence of smooth one-parameter Laplacian isospectral deformation of a closed negatively curved
manifold. For Laplace spectrum and related problems, see
\cite{DG,GW,S,Su,Vi}.

In this paper, we deal with infinite volume case and prove a finiteness
result for convex cocompact real hyperbolic 3-manifold. The main
theorem is:
\begin{thm}Let $M$ be a  convex cocompact real hyperbolic 3-manifold (which is not a solid torus)
with a length spectrum $\Lambda$ with multiplicity. Suppose $\pi_1(M)$ does not have any free factor, i.e., $\pi_1(M)$ cannot be represented as $G* F_n$ where $F_n,\ n\geq 1$ is a free group of $n$-generators. Then there exist
only finite number of convex cocompact hyperbolic 3-manifolds homotopy equivalent to
$M$ with the length spectrum $\Lambda$.
\end{thm}

For convex cocompact boundary-incompressible hyperbolic 3-manifold
case, it was treated in \cite{Kim} but the proof was incomplete. In this paper we give a complete
proof for  boundary incompressible case in Theorem \ref{incompressible-bis} and a generalization to  boundary compressible case in our main theorem.

For reader's convenience, we outline the proof of the main theorem.
Fix a convex cocompact hyperbolic 3-manifold $M$. We argue by
contradiction considering a sequence of hyperbolic 3-manifolds
homotopy equivalent to $M$ with a fixed length spectrum. The basic
idea is to find an algebraically convergent subsequence to derive a
contradiction. This process reduces the unmarked length spectrum
problem to a marked length spectrum one. In Theorem
\ref{incompressible-bis}, we show that given a discrete set
$\Lambda$, there are only finitely many convex cocompact hyperbolic
3-manifolds homotopy equivalent to $M$ with incompressible boundary,
whose length spectrum without multiplicity is contained in
$\Lambda$. This follows from Theorem \ref{ca} that one can find an
algebraically convergent subsequence if there is a lower bound for
the injectivity radius in  boundary incompressible case.

For general case in the sense that the boundary might be
compressible,  we decompose the manifold $M$ into the union of
incompressible core $\cup M_j$ connected by 1-handles where each
$M_j$ has incompressible boundary.  If $N_i$ is an infinite sequence
of convex cocompact hyperbolic 3-manifolds homotopy equivalent to
$M$ and with a fixed length spectrum with multiplicity, using
Theorem \ref{incompressible-bis}, we can assume that the covering
manifolds $N_i^j$ corresponding to $M_j$ are all isometric to each
other for a fixed $j$. Next if the lengths of the 1-handles attached are
all bounded, then by Theorem \ref{coreinject}, it is finished. If
the length of some handle goes to infinity, we argue that $N_i$ do
not have the same length spectrum with multiplicity. The general case
is treated in Thereom \ref{nonhandle}.

In section \ref{rankone}, we treat the similar result for the convex
cocompact surface group representations in rank one semisimple Lie
groups and obtain the following. 
\begin{thm}\label{main2}
Let $G$ be a semisimple real Lie group  of rank one of noncompact
type.
 Fix $\Lambda$ a discrete set
of positive real numbers, and a closed surface $S$ of genus $\geq
2$. Then the set of  convex cocompact representations
$\rho:\pi_1(S)\ra G$ with $\Lambda_\rho=\Lambda$ is finite up to
conjugacy  and the change of marking. Here $\Lambda_\rho$ is the set of translation lengths of
$\rho(\pi_1(S))$.
\end{thm}

We hope that we can generalize the argument to non-surface group case  in any
semisimple Lie groups in the near future.
\vskip .1 in
{\bf Acknowledgements.} The authors are grateful to R. Canary for
pointing out the reference  \cite{Canary1} and to anonymous referees for pointing out some inaccuracies in the earlier version, specially  Chris Croke for having pointed out a mistake in a previous version of this paper.

\section{preliminaries}\label{pre}
A real hyperbolic manifold is a locally symmetric Riemannian
manifold with constant negative sectional curvature, which is of
the form $H^n_\br/\Gamma$ where $\Gamma$ is a torsion free
discrete subgroup of $Iso(H^n_\br)$, the isometry group of $H^n_\br$. $H^n_\br$ is topologically an
open unit ball in $\br^n$ and one can compactify it by attaching
$S^{n-1}$ which is called the sphere at infinity. $Iso(H^n_\br)$
naturally acts on $S^{n-1}$ as conformal maps. A {\sl limit set}
$L_\Gamma$ of $\Gamma$ is $S^{n-1}\cap \overline{\Gamma x_0}$
where $x_0\in H^n_\br$ is a base point and the closure is with
respect to the Euclidean topology in the unit ball.
$\Omega(\Gamma)=S^{n-1}\setminus L_\Gamma$ is called the domain of
discontinuity which is the largest $\Gamma$-invariant open set of
$S^{n-1}$ on which $\Gamma$ acts properly discontinuously. For
$n=3$, $S^2=\hat \bc$ and if $\Gamma$ is not abelian,
$\Omega(\Gamma)$ inherits a hyperbolic metric. In this case,
$\Omega(\Gamma)/\Gamma$ is called the {\bf conformal boundary} of
$H^3_\br/\Gamma$ which is a hyperbolic surface.

A convex hull $CH(\Gamma)$ of $\Gamma$ is the smallest closed convex set
in $H^n_\br$ which is invariant under $\Gamma$.
$CH(\Gamma)/\Gamma$ is called the convex core $C(\Gamma)$ of
$H^n_\br/\Gamma$. If  $C(\Gamma)$ is compact, $\Gamma$
is called convex cocompact. For $n=3$, the boundary of the convex core
inherits a hyperbolic metric with respect to the path metric on
it. For any point $x\in H^n_\br$, there is a closest point in
$CH(\Gamma)$ since $CH(\Gamma)$ is a closed convex set, so there
exists a natural map
$$r:\Omega(\Gamma)/\Gamma \ra \partial C(\Gamma),$$ called the nearest
point retraction from the conformal boundary to the convex core
boundary. Using this map, one can easily see that $C(\Gamma)$ and
$(H^n_\br\cup \Omega(\Gamma))/\Gamma$ are homeomorphic if $\Gamma$
is convex cocompact and Zariski dense. In fact, $C(\Gamma)$ is a deformation retract
of $H^n_\br/\Gamma$. This is true for general negatively curved
manifolds. If $\partial(H^3_\br/\Gamma\cup \Omega(\Gamma)/\Gamma)$
is incompressible, this retraction map is $K$-Lipschitz
for some universal $K$ independent of the manifold \cite{Sull}. But
if the manifold has compressible boundary, such a universal constant
does not exist.

\begin{Def}\label{inj}Let $M$ be a Riemannian manifold without boundary. For $x$, the
injectivity radius $inj(x)$ of $x$ is defined to be the supremum of
$r>0$ such that the metric ball $B_r(x)$ in $M$ is isometric to the
$r$-ball in the universal cover $\tilde M$. For a compact convex subset $A\subset
M$, the injectivity radius of $A$ is defined to be the maximum of $r>0$ such that a finite number of $r$-balls $\overline B_r(x_i)$ (homeomorphic to the closed $r$-ball in the universal cover) contained in $A$ whose union, in such a way that $d(x_i, x_j)> r$,  $\cup\overline B_r$ is homotopy equivalent to $A$, and together with a finite number of $ r$-balls $B_r(y_j)$, not necessarily contained in $A$, in such a way that $d(y_i, y_j)>r, d(x_i, y_j)>r$,
can cover $A$.
\end{Def} The  injectivity radius of the subset of the manifold is defined in the way that if it has a long
thin compression 1-handle, then the injectivity radius is small. We will use this definition to deal with the 3-manifold with compressible boundary, specially the convex core of such a 3-manifold. See section \ref{3-manifold1}.
\begin{Def} Let $(M_i,\omega_i)$ be a sequence of hyperbolic
manifolds with orthonormal base frames. Then $(M_i,\omega_i)$ converge
geometrically to $(M,\omega)$ if and only if, for each compact
submanifold $K \subset M$ containing the base frame $\omega$, there
are smooth embeddings $f_i:K \ra M_i$, defined for all sufficiently
large $i$, such that $f_i$ sends $\omega$ to $\omega_i$ and $f_i$
tends to an isometry in $C^\infty$ topology, i.e., all the
derivatives of any order converge  uniformly on $K$ to the
derivatives of the same order of an isometry. See \cite{BP}.
\end{Def}
In practice, we use $(M_i, x_i)$ where $x_i$ is a base point of the frame $w_i$ since the set of base frames at $x_i$ is compact. This is a sequence of pointed
manifolds instead of a sequence of manifolds with base frames.

It is known that such $M$ exists if the injectivity radius at the
base point of the base frame is bounded below for all $i$. See
\cite{Mcc, BP}. For general Riemannian metrics, {\bf Gromov's
compactness theorem} states that: the set of isometry classes of
closed Riemannian $n$-manifolds with uniformly bounded curvatures,
diameters bounded above, and volumes bounded below, is precompact in
the $C^{1,\alpha}$ topology for any $\alpha<1$. More precisely, for
any sequence of Riemannian $n$-manifolds $(M_i,g_i)$ satisfying the
above conditions, there exist a subsequence, also called
$(M_i,g_i)$, and diffeomorphisms $\phi_i:M_\infty \ra M_i$ so that
the metrics $\phi_i^*g_i$ converge in the $C^{1,\alpha}$ topology to
a limit metric $g_\infty$ on $M_\infty$.

A compact irreducible 3-manifold $M$ has an {\bf incompressible
core}, which is a collection $\{M_1,\cdots,M_n\}$ of submanifolds of
$M$ such that $M$ is obtained from this collection by adding
1-handles and each $M_i$ has incomressible boundary, see \cite{MM}.
In this paper, we deal with 3-manifolds with almost incompressible boundary in the sense of McCullough \cite{McC}, i.e., $\pi_1(M)$ does not have
any free group factor.

A simple closed curve in a boundary $\partial M$ of $M$ is called a
{\bf meridian} if it is nontrivial in $\partial M$ but trivial in
$M$. By loop theorem, it bounds a compression disk. The following
lemma is practical. Before we prove the lemma, we need some
terminologies.

Any topological annulus is conformally equivalent to a Euclidean
annulus, i.e., an annulus bounded by two concentric circles. The
modulus of an annulus $A$ bouned by concentric circles of radius
$r_2>r_1$ is
$$mod(A)=\frac{1}{2\pi}\log(\frac{r_2}{r_1}).$$
In \cite{Sugawa} (Theorem 5.2), it is shown that if a conformal boundary
contains a meridian of length $L \leq 1$, then it contains a
topological annulus in the universal cover with modulus $C/L$ where $C$ can be taken as
$\frac{\pi}{\sqrt{e}}$.
\begin{lemma}\label{meridian}For a given $\epsilon$ there is  $K(\epsilon)$ such that $K(\epsilon)\ra
\infty$ as $\epsilon \ra 0$ with the following conditions. Let $M$
be a hyperbolic 3-manifold with a compressible boundary. Let $m$
be a meridian which bounds a compression disk $D$. If the length
of the meridian in a conformal boundary is $\epsilon$, then the
length of the geodesic intersecting $D$ transversely is bigger
than $K(\epsilon)$.
\end{lemma}
\begin{pf}Suppose $\epsilon$ is small so that the modulus of a
topological annulus $R$ whose core is a meridian $m$, is
$C/\epsilon>0.6$. Lift it to the universal cover of $M$. Since $m$
is a meridan, $R$ lifts to again an annulus $\tilde R$ in $\hat
\bc$. Then by \cite{HLM}, there is a Euclidean annulus $A$ inside
$\tilde R$ with modulus
$$mod(A)\geq mod(\tilde R)-\frac{1}{\pi}\log2(1+\sqrt 2)\geq mod(\tilde R)-0.502.$$
Let $CH(A)$ be the convex hull of the annulus $A$. If a geodesic
$\gamma$ intersects the compression disk $D$ transversely, $\tilde
\gamma$ should pass through $CH(A)$, and so its length is at least
$2\pi mod(A)$. If $\epsilon$ tends to zero, $mod(\tilde R)$ tends
to infinity, so does $2\pi mod(A)$.
\end{pf}
Once there is a lower bound for the lengths of meridians, one can
compare the lengths of curves in  conformal and convex core
boundary. See \cite{BC}.  From now on we drop the subscript $\br$ for simplicity to denote the real hyperbolic 3-space $H^3$.
\begin{thm}
\label{B-C} For any $\epsilon >0$, there is a constant $K>0$
depending only on $\epsilon$ with the following conditions. Let
$\Gamma$ be a finitely generated Kleinian group without torsion
such that the shortest meridian length is greater than $\epsilon$.
Let $C(\Gamma)$ be the convex core of $H^3/\Gamma$, and consider
the nearest point retraction $r: \Omega_\Gamma/\Gamma \rightarrow
\partial C(\Gamma)$. Then $r$ is $K$-Lipschitz and has a
homotopically inverse $K$-Lipschitz map.
\end{thm}

The bottom spectrum $\lambda_0$ of Laplacian is the smallest
eigenvalue of the Laplacian over the set of compactly supported
smooth functions, $C^\infty_0(N)$, which is equal to
$$\inf_{f\in C^\infty_0(N)}(\frac{\int_N |\nabla f|^2}{\int_N
f^2}).$$ Equivalently it is the largest value of $\lambda$ for
which there exists a positive $C^\infty$ function $f$ on $N$ such
that $\triangle f+ \lambda f=0$. Then it is clear that if $\tilde
N$ covers $N$, then $\lambda_0(N)\leq \lambda_0(\tilde N)$.
Furthermore by Buser \cite{Bu}, see \cite{Ca1} for a complete
argument, it is known that
\begin{Prop}\label{buser}Let $N$ be an infinite volume, pinched negatively curved complete Riemannian
manifold. Then $$ \lambda_0 \leq R \frac{Vol\partial C(N)}{Vol
C(N)},$$ where $R$ depends only on the dimension of the manifold
and the curvature bounds.
\end{Prop}
\begin{pf}
If the Ricci curvature of a complete Riemannian $n$-manifold $N$
is bounded below by $-(n-1)\kappa^2$, then $\lambda_0 \leq R\kappa
h(N)$ where $R$ depends only on $n$ and $h(N)$ is the Cheeger
constant which is defined to be the infimum, over all compact
$n$-submanifolds $A$ of $N$, of $\frac{Vol(\partial A)}{Vol(A)}$.
From this, the claim follows.
\end{pf}
For a similar result, see \cite{BuC}. For convex cocompact locally
symmetric rank 1 manifold $X/\Gamma$, it is known that
\begin{lemma}\label{bottom}$D(h-D)\leq \lambda_0(\Gamma)$ where $D$ is the
Hausdorff dimension of the limit set $L_\Gamma$ and $h$ is the
Hausdorff dimension of the boundary of $X$.
\end{lemma}
\begin{pf}The function defined by $u(x)=\int
e^{-DB(x,\theta)}d\mu_0(\theta)$ has $\triangle u=D(h-D) u$, and
so $D(h-D)\leq \lambda_0$.
\end{pf}
For convex cocompact pinched negatively curved manifold $M$ it is
known that the Hausdorff dimension of the limit set is equal to
$$\lim_{R\ra \infty} \frac{\log \#\{\gamma|l_M(\gamma)
\leq R\}}{R},$$ where $l_M(\gamma)$ is the length of the geodesic representative
in the free homotopy class of $\gamma$. Also the critical exponent of the Poincar\'e
series is equal to the Hausdorff dimension of the limit set, see
\cite{Y}. So length spectrum determines the Hausdorff dimension of
the limit set. For a Hadamard manifold $X$, if $\gamma$ is an isometry acting
on $X$, the {\sl translation length} $l_X(\gamma)$ of $\gamma$ is defined by
$$\inf_{x\in X} d(x,\gamma x).$$
From these facts we obtain:
\begin{thm}\label{volume} Let $N=H^3/\Gamma$ be a convex cocompact, infinite volume
hyperbolic 3-manifold, and let $D$ denote the Hausdorff dimension of
the limit set $L_\Gamma$.
Then $D(2-D)\leq \lambda_0\leq C \frac{|\chi(\partial
C(N))|}{vol(C(N))}$ for some universal constant $C$.
\end{thm}
\begin{pf}
The inequality follows  from Proposition \ref{buser} and Lemma
\ref{bottom}.
\end{pf}

\section{Some examples}
%
Mahler's criterion \cite{cha} states that if $G$ is a Lie group and
$U$ is a neighborhood of the identity $e$, then the set of subgroups
$\Gamma$ of $G$ such that $\Gamma \cap U=\{e\}$ is a compact set
$\cal U_G$ with respect to the Chabauty topology. To see the
relation between Chabauty topology and geometric topology in Gromov
sense, see \cite{BP}.  First observe the following useful fact
\cite{Kim1}, which will be used throughout the paper.
\begin{lemma}\label{determining}Let $\Gamma$ be a fixed finitely generated group.
Let $\rho:\Gamma \ra G$ be a Zariski dense
representation into a real semi-simple Lie group of rank 1. Then
there is a finite generating set $\{\gamma_1,\cdots,\gamma_k\}$
depending only on $\Gamma$ so that the representation is
determined up to conjugacy by the marked length spectrum on this
finite generating set.
\end{lemma}
For a similar result in surface, see \cite{BCo}. One can easily
deduce the following.
\begin{Prop}\label{rank1}Let $G$ be a semisimple Lie group of rank 1 with $X=G/K$ and $\Lambda$ a discrete
set of positive real numbers.  Then the set of convex cocompact manifolds $\Gamma
\backslash G / K$ with
 diameter for the convex core bounded above by $R$ and with a
 length spectrum   {\bf without  multiplicity} contained in $\Lambda$, is finite
up to isometry.
\end{Prop}
\begin{pf}One can find a neighborhood $U$ of $e$ so that for any
$\gamma \in U$, $l(\gamma)< r_0$ where $r_0$ is the smallest number
in $\Lambda$. Suppose there exist infinitely many non-isometric
convex cocompact manifolds $\Gamma_i \backslash G / K$ with its
length spectrum contained in $\Lambda$. Since $\Gamma_i\in \cal
U_G$, one can find $\Gamma$ so that $\Gamma_i\ra\Gamma$ in Chabauty
topology after passing to a subsequence. In a geometric term
\cite{BP}, it implies that, for $x_i\in C(X/\Gamma_i)\subset B(x_i,R)$,
$$(X/\Gamma_i,x_i)\ra (X/\Gamma,x)$$ in Gromov sense, i.e., for any ball
$B(x,r)\subset X/\Gamma$, there is a smooth embedding $f_i:B(x,r)\ra
X/\Gamma_i$ for large $i$ so that $f_i(x)=x_i$ and $f_i$ tends to an
isometry in $C^\infty$ topology. Since $C(X/\Gamma_i)\subset B(x_i,R)$, for given small $\epsilon>0$,  there exists $i_0$ such that  for all $i>i_0$,
$C(X/\Gamma_i)\subset f_i(B(x,R+\epsilon))$. In negatively curved Hadamard
manifolds, the convex core is a strong deformation retract of a
convex cocompact manifold via shortest distance retraction to the
convex core.

 Then using $f_j\circ
f_i^{-1}:f_i(B(x,R+\epsilon))\ra f_j(B(x,R+\epsilon))$ one can see that
$C(X/\Gamma_i)$ are all homotopy equivalent to each other and so
$\Gamma_i$ are all isomorphic to each other. Choose an isomorphism
$\rho_i:\Gamma' \ra \Gamma_i$ induced from the map $f_i:B(x,R+\epsilon)\ra
X/\Gamma_i$, which induces a representation $\rho_i:\Gamma'\ra G$,
where $\Gamma'=\pi_1(x,B(x,R+\epsilon))$ is isomorphic to a subgroup of $\Gamma$.
The surjectivity of $\rho_i$ is evident since any element in $\Gamma_i$ is represented by a loop based at $x_i$, which is contained in $C(X/\Gamma_i)$.
Then the loop is realized by the image of some loop in $B(x,R+\epsilon)$ based at $x$ under $f_i$.  The injectivity of $\rho_i$ follows from Gauss-Bonnet theorem $$\int_D K dA+\int_{\partial D} k_g ds= 2\pi \chi(D)=2\pi,$$ where $K$ is Gaussin curvature of a disc $D$ and $k_g$ is a geodesic curvature of $\partial D$.
Suppose $f_i(\gamma)$ represents a trivial element in $\pi_1(x_i, C(X/\Gamma_i))$ for some geodesic loop $\gamma$ based at $x$, representing a non-trivial element in $\pi_1(x, B(x,R+\epsilon))$. Choose disc $D_i$ bounding $f_i(\gamma)$, whose Gaussian curvature $<-\delta$.
Then for large $i>i_0$, since $f_i$ is  almost an isometry, $f_i(\gamma)$ is  almost a geodesic, hence $k_g$ along $\partial D_i$ is almost zero with
the interior angle at $x_i$ is $\theta_i$.
Hence the formula reads
$$-\delta Area(D_i) > 2\pi - (\pi- \theta_i)= \pi+\theta_i,$$ which gives
$Area(D_i) < \frac{\pi+\theta_i}{-\delta}<0$, a contradiction.

 Since $f_i$
is almost an isometry, it follows that $l_{\rho_i}(\gamma),\gamma\in
\Gamma'$ are bounded above. Since $\Lambda$ is a spectrum of a
convex cocompact manifold, it is discrete. Then by passing to a
subsequence, one can assume that $l_{\rho_i}(\gamma)$ is constant
for a fixed $\gamma$. Fix a finite determining set
$\{\gamma_1,\cdots,\gamma_k\}$ from Lemma \ref{determining}. After
passing to a subsequence we may assume that
$l_{\rho_i}(\gamma_j)=l_{\rho_l}(\gamma_j),j=1,\cdots,k$ and for all
$i,l$. But then $X/\Gamma_i$ are all isometric to each other, which
contradicts  the choice of $\Gamma_i$.
\end{pf}
Note that a priori we did not assume that $\Gamma_i$ are isomorphic
each other. That $\Gamma_i$ are isomorphic each other follows during
the course of the proof. Also we did not use the multiplicity in the
spectrum. In fact, the proof works even we just assume that the
length spectrum without multiplicity is contained in a fixed
discrete set, not necessarily the same. The issue for bounded
diameter of convex core can be dealt with in real hyperbolic
manifold as follows.

\begin{co}\label{incompressible}Let $M$ be a hyperbolic 3-manifold with
incompressible boundary and with a discrete length spectrum. Then
there are finitely many convex cocompact hyperbolic 3-manifolds, up
to isometry, homotopy equivalent to $M$ with the length spectrum
$\Lambda(M)$ with multiplicity.
\end{co}
\begin{pf}For convex cocompact manifold, $\Lambda(M)$ determines
the critical exponent of Poincar\'e series which is the Hausdorff
dimension of the limit set, $D$ as observed in section \ref{pre}.
Then Theorem \ref{volume} implies that the volume of the convex core
is uniformly bounded above. Also for boundary incompressible
manifold, the injectivity radius on convex core is uniformly bounded
below by the smallest number in $\Lambda(M)$. So the diameter of the
convex core has a uniform upper bound. Then the conclusion follows
from Proposition \ref{rank1}.
\end{pf}
In the next section, we will prove a theorem {\bf without
multiplicity} condition using Kleinian group theory.
\section{Hyperbolic 3-manifold with Incompressible boundary}
Let $M$ be a fixed hyperbolic 3-manifold. Then $AH(M)$ denotes the
set of marked hyperbolic 3-manifolds homotopy equivalent to $M$.
This is the set of discrete faithful representations from $\pi_1(M)$
to $PSL(2,\bc)$ up to conjugacy. In this reason one denotes it by
$AH(\pi_1(M))$. Hence an element in $AH(M)$ can be written as
$N=(f,M,N)$ with
$$f:M\ra N$$ where $f$ is a marking and $N$ is a hyperbolic 3-manifold homotopy
equivalent to $M$.

From now on, we describe some convergent theorem for hyperbolic 3-manifolds.
 Indeed, from \cite{Canary1}, one can extract an algebraically convergent subsequence if a sequence $N_i$
has a lower bound for injectivity radius. For the readers who are not familiar
with 3-dimensional hyperbolic geometry, we give some ideas.

The window $W$ of $M$ consists of the $I$-bundle components of the
characteristic submanifold $\Sigma(M)$ together with a thickened
neighborhood of every essential annulus in $\partial
\Sigma(M)\setminus \partial M$ which is not the boundary of an
$I$-bundle component of $\Sigma(M)$.  The window itself is an
$I$-bundle over a surface $w$, which is called the window base.

Thurston \cite{thur} showed that if $\Gamma$ is any subgroup of
$\pi_1(M)$ which is conjugate to the fundamental group of a
component of $M\setminus W$ whose closure is not a thickened torus,
then a restriction of any sequence in $AH(\pi_1(M))$ to $AH(\Gamma)$
has a convergent subsequence. Using this they proved the following
\cite{Canary1}:
\begin{thm}\label{ca}Let $\{\rho_i\}$ be a sequence in $AH(\pi_1(M))$. We may
then find a subsequence $\{\rho_j\}$, a sequence of elements
$\{\phi_j\}$ of $Out(\pi_1(M))$, and a collection $x$ of disjoint,
non-parallel, homotopically non-trivial simple closed curves in the
window base $w$ such that if $\Gamma$ is any subgroup of $\pi_1(M)$
which is conjugate to the fundamental group of a component of
$M\setminus X$ whose closure is not a thickened torus, where $X$ is
the total space of the $I$-bundle over $x$, then
$\{\rho_j\circ\phi_j|_\Gamma\}$ converges in $AH(\Gamma)$. Moreover,
if $c$ is a curve in $x$, then $\{l_{\rho_j\circ\phi_j}(c)\}$
converges to $0$.

Specially if $\rho_i\in AH(\pi_1(M))$ is a sequence with
lower bound on injectivity radius, then there exists a subsequence
$\rho_j$ and $\phi_j\in Out(\pi_1(M))$ so that $\rho_j\circ \phi_j$
converges algebraically.
\end{thm}
The idea is as follows. By Thurston, the restrictions of the
representations to the complementary components of the window have
convergent subsequences, and the lengths of the window boundaries
are bounded. For each $\rho_i$, represent $w$ as a pleated surface
with geodesic boundary. After passing to a subsequence and
remarkings, either the hyperbolic structures induced by pleated
surfaces converge or develop cusps. In the latter case we cut along
the curves which become cusps (these are the family $x$), and argue
that the representations restricted to the remaining components
converge up to subsequence. Finally, one reglues along the window
boundaries which did not converge to cusps, and after passing to
further subsequence obtain such a convergent subsequence.

Now we prove the following theorem.
\begin{thm}\label{incompressible-bis} Let $M$
be a hyperbolic 3-manifold with incompressible boundary and
$\Lambda$ a discrete subset of $\mathbb R$. Then there are finitely
many hyperbolic 3-manifolds, up to isometry, homotopy equivalent to
$M$ so that the length spectrum {\bf without multiplicity} is
contained in $\Lambda$.
\end{thm}
\begin{pf}
Let $N_i=H^3/\rho_i(\pi_1(M))$ be  infinitely many non-isometric
such hyperbolic manifolds whose length spectrum without multiplicity
is contained in $\Lambda$.

Suppose after passing to a subsequence, we can obtain a convergent
sequence $\{\rho_i\circ (\phi_i)_*\}$ converging to $\rho$ up to
changing the marking by $\phi_i$, which is guaranteed by Theorem
\ref{ca}.

We have
$$l_{\rho_i\circ (\phi_i)_*}(\gamma)\ra l_\rho(\gamma).$$
Since $\{l_{\rho_i\circ (\phi_i)_*}(\gamma)\subset \Lambda\}$ is
discrete, after passing to a subsequence, we conclude that
$l_{\rho_i\circ (\phi_i)_*}(\gamma)=l_\rho(\gamma)$ for a fixed
$\gamma$. Then one can repeat this for a finite set of
$\gamma_1,\cdots,\gamma_n$ in $\pi_1(M)$. By Lemma
\ref{determining} or by \cite{Kim2}, a representation is determined up to conjugacy by
the translation length of a finitely many elements in $\pi_1(M)$. So
we conclude that $\rho_i \circ (\phi_i)_*$ represent isometric
manifolds, which is a contradiction.
\end{pf}

For future purpose, we record the following theorem as well
\cite{Mor}.
\begin{thm}\label{convergence}Let $\rho_i:\Gamma \ra
PSL(2,\bc)$ be a sequence of representations such that
$l_{\rho_i}(\gamma)$ is uniformly bounded for each
$\gamma\in\Gamma$. Then there is a convergent subsequence.
\end{thm}

\section{Hyperbolic 3-manifold whose injectivity radius
is bounded from below on convex core}\label{3-manifold1} In this
section we discard the assumption that $M$ has an incompressible
boundary, nonetheless we stick to the case where manifolds are convex
cocompact.

Now we state the theorem. Note that when $M$ has a compressible
boundary, the smallest geodesic length of $M$ does not determine
the injectivity radius on the convex core due to the compression
disk.
\begin{thm}\label{coreinject}
Suppose $M$ is a convex cocompact hyperbolic 3-manifold.
Then the set of convex cocompact hyperbolic 3-manifolds homotopy
equivalent to $M$ with the length spectrum $\Lambda(M)$ and with
the injectivity radius of the convex core being bounded from
below in the sense of  Definition \ref{inj}, is finite up to isometry.
\end{thm}
\begin{pf}Suppose there are infinitely
many non-isometric convex cocompact hyperbolic 3-manifolds $N_i$
 with $\Lambda(M)$. We will derive a
contradiction as before.
The unmarked length spectrum determines a critical exponent of
Poincar\'e series which is equal to $D$ in convex cocompact case.
Then by Theorem \ref{volume}, the volume of the convex core has an
upper bound. Since the injectivity radius of the convex core is
bounded from below, its diameter has an upper bound. This can be easily seen as follows.  By the defition of injectivity radius, $\cup B_r(x_i)  \cup B_r(y_j)$ covers $C(N_i)$ with the condition  $d(x_i, x_j)>r, d(x_i, y_j)>r, d(y_i, y_j)>r$.  Set $\text{vol}(B_r(x))=\epsilon$. Since $x_i$ and $y_j$ are $r$-separated one can easily estimate
$$\text{ vol}(\cup B_r(x_i)  \cup B_r(y_j))\geq f(\text {number of balls}) \epsilon$$ where $f(\text {number of balls})\ra \infty$ as the number of
balls goes to infinity.  Hence the number of balls to cover $C(N_i)$ is uniformly bounded and hence the diameter is uniformly bounded.

Take a base point $p_i \in C(N_i)$. Then the injectivity radius at
$p_i$ is bounded below by the assumption. Then the sequence
$(N_i,p_i)$ converge to $(N,p)$ in geometric topology where $N$ is
a complete hyperbolic 3-manifold, \cite{Mcc}. Take a number $R$
which is greater than the diameter of $C(N_i)$ for all $i$.  By
the definition of geometric convergence, there exist a ball
$B(p,R)\subset N$ and a smooth embedding $f_i:B(p,R)\ra N_i$ whose
quasiisometric constant converges to 1 as $i$ tends to $\infty$,
and so $f_i(B(p,R))$ contains $C(N_i)$ and $B(p,R)$ is
homeomorphic to a neighborhood of $C(N_i)$ for all large $i$. This can be easily seen as follows. As in the proof of Proposition \ref{rank1}, one can show
$(f_i)_*:\pi_1(p, B(p,R))\ra \pi_1(p_i, f_i(B(p,R)))$ is an isomorphism. And since $C(N_i)\subset f_i(B(p,R))$, the nearest point retraction of $f_i(B(p,R))$ to $C(N_i)$ induces a homotopy equivalence between $B(p,R)$ and $C(N_i)$. Since a neighborhood of $C(N_i)$ is homeomorphic to $N_i$,   we conclude that $N_i$ are
all homeomorphic to $B(p,R)$. Note that $B(p,R)$ is homotopy
equivalent to $M$.

Now we can argue as in the proof of Proposition \ref{rank1}. Let $\gamma_1,\cdots,\gamma_k\in \pi_1(M)$ be
a determining set, i.e., the marked length spectrum of these
elements determines the representation up to conjugacy.

Now give a marking $f_i:B(p,R)\ra N_i$ to $N_i$, which will induce
a representation $\rho_i:\pi_1(M)\ra PSL(2,\bc)$ corresponding to
$N_i$. Then $l_{N_i}(\gamma_1)\ra l_N(\gamma_1)$. Since
$\Lambda(M)$ is discrete, after passing to a subsequence, we may
assume that $l_{N_i}(\gamma_1)=l_N(\gamma_1)$. Doing this
$k$-times, we get $l_{N_i}(\gamma_j)=l_N(\gamma_j)$ for all $i$
and $j=1,\cdots,k$. This shows that all $N_i$ are isometric to each
other, which is a contradiction.
\end{pf}

\section{Hyperbolic 3-manifold without free factors}\label{3-manifold2}
In this section we prove our main result.
Let $M$ be a convex cocompact hyperbolic 3-manifold with
compressible boundary which does not have a handlebody factor, i.e.,  it has
incompressible cores $M_i$ such that a unique 1-handle $h^{i}$ connects $M_i$ to $M_{i+1}$. In terms of the fundamental group, $\pi_1(M)$ does not have a free factor $F_n$ where $F_n$ is a free group with $n$ generators, see section \ref{pre}.

\begin{thm}\label{nonhandle}Let $M$ be a  convex cocompact real hyperbolic 3-manifold (which is not a solid torus)
with a length spectrum $\Lambda$ with multiplicity. Suppose $\pi_1(M)$ does not have any free factor, i.e., $\pi_1(M)$ cannot be represented as $G* F_n$ where $F_n,\ n\geq 1$ is a free group of $n$-generators. Then there exists only finitely many non-isometric
convex cocompact hyperbolic 3-manifolds which are homotopy
equivalent to $M$ and with the length spectrum $\Lambda(M)$ with
multiplicity.
\end{thm}
\begin{pf}Suppose  there is an infinite sequence of mutually non-isometric
convex cocompact hyperbolic 3-manifolds $N_i=H^3/\Gamma_i$ with
length spectrum $\Lambda(M)$. We will first deal with the case where
there is only one 1-handle to give a better understanding to the
reader. First assume that the incompressible core of $M$ is $M_1 \cup
M_2$ and $M$ is obtained from this by adding an 1-handle $h$ glued
to boundaries of $M_1$ and $M_2$ ($M_1=M_2$ is not allowed).
\begin{figure}
$$\includegraphics[width=.85\linewidth]{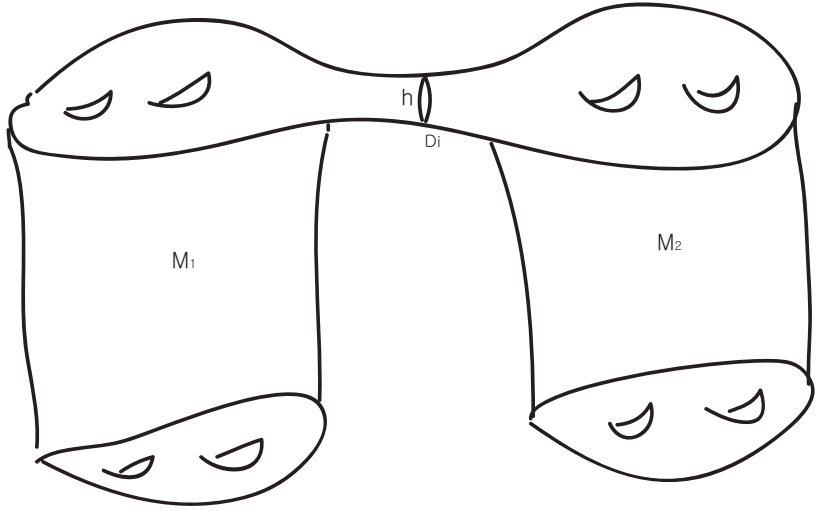}$$
\caption{$M$ consists of $M_1$, $M_2$ and one handle $h$.}
\end{figure}
Then $\pi_1(M)=\pi_1(M_1)*\pi_1(M_2)$. Accordingly
$\Gamma_i=\Gamma_i^1 * \Gamma_i^2$. Let $N_i^j$ be a cover of $N_i$
corresponding to $\pi_1(M_j), j=1,2$. Since $M_1$ has incompressible
boundary, $N_i^1$ has incompressible boundary also. Note that
$N_i^1$ is convex cocompact by covering theorem \cite{Ca2}.
Now we apply  Theorem \ref{incompressible-bis}. Since
$\Lambda(N_i^j) \subset \Lambda(N_i)= \Lambda(M)$,
 after passing to a subsequence, we may assume that all $N_i^1$ are
isometric to each other. The same thing is true for $N_i^2$.

 By using a
covering map $P_j:N_i^j \ra N_i$ (we drop $i$ in $P_j$ for
simplicity), $N_i$ has two parts $P_1(C(N_i^1))$ and $P_2(C(N_i^2))$
homotopy equivalent to $M_1$ and $M_2$ respectively, such that
$N_i^j$ are isometric to each other for all $i$ and fixed $j=1,2$.
 Note that $P_1(C(N_i^1))$ and $P_2(C(N_i^2))$ might  not be disjoint.  Since the convex
hull is the smallest convex set containing all the closed geodesics, it is clear that $P_j(C(N_i^j))\subset C(N_i)$.
  For
each closed geodesic $\gamma \subset N_i^j$, $P_j:\gamma \ra
P_j(\gamma)$ is a homeomorphism of degree 1 viewed as a self map
from a circle to itself since $N_i^j$ is a covering corresponding to
$\pi_1(M_j)$. In particular geodesic lengths coming from
$P_1(C(N_i^1))$ are all the same for all $i$, and the same thing is
true for $P_2(C(N_i^2))$.

Since $C(N_i)$ is homotopy equivalent to $M$, decompose $C(N_i)=C_i^1 \cup h \cup C_i^2$ such that $C_i^1$ contains $P_1(C(N_i^1))$, convex and  homotopy equivalent to $M_1$,
$C_i^2$ contains $P_2(C(N_i^2))$, convex and  homotopy equivalent to $M_2$, and $h$ is a  1-handle joining them. Note here that $C_i^1, h$ and $C_i^2$ are not mutually disjoint in general.
Since $C_i^j$ has irreducible boundary and the length spectrum is fixed for it, the injectivity radius is bounded below on $C_i^j$. Hence by Theorem \ref{volume} again, the diameter of $C_i^j$ is bounded.

If the injectivity radius of $C(N_i)$ in the sense of Definition 1 is bounded below, then we are done by Theorem \ref{coreinject}, hence assume that the injectivity radius goes to zero as $i\ra\infty$.
This is possible only along compression disks which are the dual core of attached one-handles since the injectivity radius is bounded below on $C_i^j$.

Let the length of the meridian, which is the boundary of compression disk $D_i$, tend to zero as $i\ra\infty$. See Figure 1.
Then by Theorem \ref{B-C}, the meridian length in corresponding conformal boundary tends to zero also. By the proof of Lemma \ref{meridian}, there exists a solid
cylinder neighborhood of $D_i$ whose height tends to infinity. Hence $h$ is non-empty and  the length of any geodesic crossing $h$  goes to infinity.


Find a geodesic $\gamma$ which is not in $P_1(C(N_1^1))$, neither in
$P_2(C(N_1^2))$. Then $\gamma$ must cross $h$. Take $i$ large enough
so that the length in $N_i$ of any geodesic crossing the 1-hanlde
$h$ is much greater than $l_{N_1}(\gamma)$. This is possible  by Lemma \ref{meridian} as explained above. Now by assumption, $N_i$ and $N_1$ have the same
length spectrum {\bf with multiplicity}. Note also that $P_j(C(N_i^j))$
and $P_j(C(N_1^j)), j=1,2$ have the same length spectrum since
$N_i^j$ are isometric. If $l_{N_1}(\gamma)$ is the length of a
geodesic in $N_1^j$, the multiplicity for $l_{N_1}(\gamma)$ will be
different for $N_1$ and $N_i$. If $l_{N_1}(\gamma)$ is not the
length of any geodesic in $N_1^j$, there will be no geodesic in
$N_i$ whose length is equal to $l_{N_1}(\gamma)$. In either case,
$N_1$ and $N_i$ do not have the same length spectrum. Note that we
are considering the length spectrum with {\bf multiplicity}. Hence the injectivity radius on $C(N_i)$ should remain bounded below, and we are done by Theorem \ref{coreinject}.

Now we deal with the case where there are many 1-handles.
 For a
given $M$, suppose $M=M_1\cup h^1 \cup M_2 \cup h^2 \cup \cdots \cup h^t \cup M_{t+1}$.  By
taking a cover $N_i^k$ corresponding to $M_k$, by Theorem
\ref{incompressible-bis} after passing to a subsequence, we may
assume that all $N_i^k$ are isometric to each other for a fixed $k$.
Now using the projection $P_k:N_i^k\ra N_i$, there are two cases to
consider. If the injectivity radius is bounded below on $C(N_i)$,  we
are done by Theorem \ref{coreinject}.  As before, decompose $C(N_i)=C_i^1 \cup h_i^1 \cup C_i^2 \cup h_i^2 \cup C_i^3\cdots$ such that $C_i^l$ contains $P_l(C(N_i^l))$, convex and homotopy equivalent to $M_l$.  Note that the injectivity radius on $C_i^k$ is uniformly bounded below.

If the injectivity radius on $C(N_i)$ is not bounded below, choose the first $h_i^l$ whose height tends to infinity as in the previous case.

\begin{lemma}\label{collect}Let $M'=M_1\cup h^1\cup \cdots \cup h^{l-1}\cup M_{l}$ and $N_i'$ the corresponding covering manifolds. Then $N_i'$ are isometric to each other after passing
to a subsequence.
\end{lemma}
\begin{proof}By construction, all the 1-handles  in $C(N_i')$ have bounded lengths.  Any element $\gamma$ in $\pi_1(M')=\pi_1(M_1)*\cdots *
\pi_1(M_l)$  can be written as $\gamma_1*\cdots*\gamma_l$ where base points for $\gamma_i$ can be connected by arcs going over 1-handles.
Since 1-handles have bounded lengths and $N_i^k$ are all isometric to each other, each $\gamma_i$ is realized as a geodesic which has the same length for all $i$, and the arcs connecting them have the bounded lengths.  This implies that any $\gamma$ in $\pi_1(M')$  can be realized as a geodesic which has a uniform bounded length for all $i$.  Then by Theorem \ref{convergence}, the representations $\rho_i:\pi_1(M')\ra PSL(2,\bc)$ corresponding to $N_i'$ converge to some representation $\rho$ after passing to a subsequence.
Since the length spectrum of $N_i'$ is contained in a fixed discrete length spectrum, using the same argument as in the proof of Theorem \ref{incompressible-bis} on a finitely many determining elements in $\pi_1(M')$, after passing to a subsequence, $N_i'$ are all isometric to each other.
\end{proof}

In this way collect submanifolds $M'_1,\cdots, M'_k$ so that $M=M'_1\cup h^1\cup \cdots \cup h^{k-1} \cup M_k'$ and the handles in $M'_m, m=1,\cdots,k$ have bounded lengths throughout the sequence, and the other handles are getting longer as $i\ra\infty$.
Then after passing to a subsequence, all the coverings $N'_{m,i}, m=1,\cdots,k$ of $N_i$ corresponding to $M'_m$ are isometric to each other by Lemma \ref{collect}.   Let $P_m:N'_{m,i}\ra N_i$ be a covering
map, and decompose  
$$C(N_i)=C_i^1\cup h_i^1\cup \cdots \cup h_i^{k-1} \cup C^k_i$$ where $C_i^m$ is a convex set
containing $P_m(C(N'_{m,i}))$ homotopy equivalent to $M'_m$. Note that the length of $h_i^m$ goes to infinity as $i\ra\infty$ and the length spectrum on
$C_i^m$ is fixed for all $i$ since $N'_{m,i}$ are all isometric to each other.

Now  apply the same logic as in the case where there exists only one 1-handle as follows.
Choose $\gamma$ passing through $h^1$. 
Choose $i$ large enough so that any loop passing through one of the handles has length in $N_i$ greater than
$\ell_{N_1}(\gamma)$. This is possible since all the lengths of handles $h_i^m, m=1,\cdots, k$, go to infinity.
If $\ell_{N_1}(\gamma)$ is realized as a length of some geodesic in $C^m_1$, the multiplicity in $N_1$ for
$\ell_{N_1}(\gamma)$ is greater than the multiplicity in $N_i$.
If $\ell_{N_1}(\gamma)$ is not realized as a length of any curve in $C^m_1$ for all $m=1,\cdots,k$, 
this length is not realized in $N_i$ at all. In any case, $N_1$ and $N_i$ will have different length spectrums with
multiplicity, which is a contradiction.

Hence, there is no handle whose length goes to infinity. This implies that the injectivity radius of $C(N_i)$ is bounded below. Then by Theorem \ref{coreinject}, we are done.
\end{pf}
{\bf Remark} Somehow handlebody case is difficult to handle. If two generators are getting longer and their invariant
axes are getting closer, it is hard to draw a contradiction.
\section{Rank one symmetric space case}\label{rankone}
In this section we prove the isospectral finiteness of surface group
in rank one symmetric space. It is known that one can deform a
totally real hyperbolic surface sitting in complex, quaternionic and
octonionic hyperbolic manifold \cite{Kim3}. If the deformation is
small, the group becomes convex cocompact.

Let $M$ and $N$ be Riemannian manifolds of dimension $m,n$
respectively with metric tensors $(\gamma_{\alpha\beta})$ and
$(g_{ij})$. The curvature tensor is defined by
$$R(X,Y)Z=\nabla_X\nabla_Y Z-\nabla_Y\nabla_X Z- \nabla_{[X,Y]}Z$$
and the sectional curvature spanned by orthonormal $X,Y$ is
$$K(X,Y)=\langle R(X,Y)Y, X \rangle.$$

Let $f:M\ra N$ be a $C^1$ map and using local charts
$(x^1,\cdots,x^m)$ on $M$, $(f^1,\cdots,f^n)$ on $N$, the energy
density is
$$e(f)(x)=\frac{1}{2}\gamma^{\alpha\beta}(x)g_{ij}(f(x))
\frac{\partial f^i(x)}{\partial x^\alpha}\frac{\partial
f^j(x)}{\partial x^\beta}.$$ If $\{e_1,\cdots,e_m\}$ is an
orthonormal basis of $T_xM$, then
$$e(f)(x)=\frac{1}{2}\sum \langle df(e_i), df(e_i)\rangle=\frac{1}{2}|df|^2.$$
The energy of $f$ is
$$E(f)=\int_M e(f) dvol_M.$$
A solution of the Euler-Lagrange equation for $E$ is called a
harmonic map.

The Laplace-Beltrami operator $\triangle$ is defined so that
$$\triangle f=-\sum \frac{\partial^2 f}{(\partial x^i)^2}$$ for real
valued function $f$.

Bochner type equality for a harmonic map $f:M\ra N$ is \cite{J}
$$-\triangle e(f)(x)$$$$=|\nabla df|^2 + \langle df Ric(e_i), df(e_i)
\rangle - \langle R_N(df(e_i), df(e_j))df(e_j), df(e_i) \rangle.$$
Here $\{e_i\}$ is an orthonormal basis of $T_xM$ and $$Ric(v)=\sum
R_M(v,e_i)e_i.$$

Let $(S,\sigma)$ be a hyperbolic surface and $G$ a semisimple Lie
group, $X$ the associated symmetric space. Let $\rho:\pi_1(S)\ra G$
be a representation whose image is reductive. Let $J$ be a complex
structure compatible with $\sigma$ on $(S,\sigma)$ lifted to $\tilde
S$ and $f:\tilde S\ra X$ a smooth $\rho$-equivariant map. Then
$$(df\wedge df\circ J)(u,v)=\frac{1}{2}(\langle
df(v),df(Ju)\rangle-\langle df(u),df(Jv)\rangle)$$ defines an
exterior differential 2-form on $\tilde S$, which descends to $S$.
The energy of $f$ with respect to $J$ (or $\sigma$) can be
alternatively defined by $$E(J,f)=\int_S df\wedge df\circ J.$$ Then
the energy functional $E_\rho$ on $\cal T(S)$ is defined by
$$E_\rho(J)=\inf_f E(J,f)$$ for all $\rho$-equivariant $f$.

Then there exists  a $\rho$-equivariant harmonic map $f:(\tilde
S,\tilde \sigma) \ra X$ inducing $\rho$ with energy
$E_\rho(\sigma)$, see \cite{Co}. This function $E_\rho$ is a proper
function on $\cal T(S)$ if $\rho$ is convex cocompact \cite{La}. So
there is a harmonic map $f_\rho$ and a hyperbolic metric
$\sigma_\rho$ so that
$$f_\rho:(\tilde S,\tilde \sigma_\rho)\ra X$$ minimizes the energy
for all such possible choices of $f$ and $\sigma$. This map is known
to be conformal \cite{SU}.

\begin{thm}\label{main2}
Let $G$ be a semisimple real Lie group  of rank one of noncompact
type.
 Fix $\Lambda$ a discrete set
of positive real numbers, and a closed surface $S$ of genus $\geq
2$. Then the set of  convex cocompact representations
$\rho:\pi_1(S)\ra G$ with $\Lambda_\rho=\Lambda$ is finite up to
conjugacy  and the change of marking. Here $\Lambda_\rho$ is the set of translation lengths of
$\rho(\pi_1(S))$.
\end{thm}
\begin{pf} Suppose there are infinitely many non-conjugate
representations $\rho_i:\pi_1(S)\ra G$ with
$\Lambda_{\rho_i}=\Lambda$. Choose a harmonic map $h_i:(\tilde
S,\sigma_i)\ra X$ so that the energy is the smallest among all such
$h_i$ inducing $\rho_i$ and among all hyperbolic metrics $\sigma_i$.
Then it is known that \cite{SU}, $h_p$ is conformal, i.e.,
$dh_p(e_1)$ and $dh_p(e_2)$ are orthogonal and have the same norm
for an orthonomal basis $\{e_1,e_2\}$ of $T_x(S,\sigma_p)$. Hence
$|dh_p|^2=2|dh_p(e_1)|^2$. Then by Bochner formula
$$-\frac{1}{2}\triangle|dh_p|^2=|\nabla dh_p|^2$$$$+ \langle dh_p
Ric_{\sigma_p} (e_i),dh_p(e_i)\rangle - \langle
R_X(dh_p(e_i),dh_p(e_j))dh_p(e_j),dh_p(e_i)\rangle,$$
 where
${e_1,e_2}$ is an orthonormal basis at a point of $(S,\sigma_p)$.

Since $$Ric(e_1)=R(e_1,e_1)e_1+R(e_1,e_2)e_2$$ and since $$\langle
Ric(e_1), e_1 \rangle=\langle R(e_1,e_2)e_2, e_1\rangle=-1,$$ and
similar for $Ric(e_2)$, we get $Ric(e_i)=-e_i+x_ie_j$ to have
$$-\frac{1}{2}\triangle |dh_i|^2$$
$$\geq -|dh_i|^2 -
|dh_i(e_1)|^42K_X(\frac{dh_i(e_1)}{|dh_i(e_1)|},\frac{dh_i(e_2)}{|dh_i(e_1)|}),$$
where $X$ is an associated symmetric space. Since $X$ is a rank one
symmetric space,
$K_X(\frac{dh_i(e_1)}{|dh_i(e_1)|},\frac{dh_i(e_2)}{|dh_i(e_1)|})
\leq -1$. This implies that
$$-\frac{1}{2}\triangle |dh_i|^2\geq
|dh_i|^2(|dh_i(e_k)|^2-1)$$ for $e_k,k=1,2$. If $x_0\in S$ is a
maximum point of $ |dh_i|^2$, then $$-\triangle |dh_i|^2(x_0)\leq 0,$$
so we get $$|dh_i(e_k)|^2 \leq 1,$$ which implies that
$$|dh_i(v)|\leq |v|$$ for $v\in T_{x_0}S$.
For any $x\in S$, since $$2|dh_i(x)(e_{k,x})|^2=|dh_i(x)|^2\leq
|dh_i(x_0)|^2=2|dh_i(x_0)(e_{k,x_0})|^2\leq 2$$ for $k=1,2$, we get
$$l_{\rho_i}(\gamma)\leq l_{\sigma_i}(\gamma),$$ for all $\gamma\in
\pi_1(S)$. But since $l_{\rho_i}(\gamma)\geq r_0$ where $r_0$ is the
smallest number in $\Lambda$, using Mumford compactness theorem, we
get, after changing the marking of  $(S,\sigma_i)$, $(S,\sigma_i)\ra
(S,\sigma)$. Then after changing the marking of $\rho_i$
accordingly, we get
$$l_{\rho_i}(\gamma)\leq Cl_{\sigma}(\gamma),$$ for all $\gamma\in
\pi_1(S)$.

Then one can pass to a subsequence so that $\rho_i\ra \rho$. Then
$$l_{\rho_i}(\gamma)\ra l_\rho(\gamma)$$ for all $\gamma\in
\pi_1(S)$. By discreteness of $\Lambda$, for each $\gamma$, we may
assume that $l_{\rho_i}(\gamma)= l_\rho(\gamma)$ for large $i$. We
repeat this for finitely many $\{\gamma_1,\cdots,\gamma_n\}$ which
determines the representations in $G$ up to conjugacy by Lemma
\ref{determining}. So we conclude that $\rho_i$ and $\rho$ are
conjugate for all $i$, which is forbidden by the assumption.
\end{pf}
\section{Appendix A}
The materials covered in this appendix are well-known to the
experts, specially for dimension 3 it is classical, though it is
difficult to find a right literature for higher dimension, but for
the convenience of the readers we give some proofs. The goal of this
appendix is to show that the hyperbolic metric is uniquely
determined by the metrics on the smooth convex boundary, Theorem
\ref{boundarymetric}. In this section $H^n_\br=H^n$.

A homeomorphism $\phi:X \ra Y$ is $K$-{\sl quasiconformal} if
$\phi$ has distributional first derivatives locally in $L^n$, and
$$\frac{1}{K}|\text{ det}D\phi(x)|\leq (\frac{|D\phi(v)|}{|v|})^n \leq K|\text
{det} D\phi(x)|$$ for almost every $x$ and every nonzero vector
$v\in T_xX$. It is well-known from the techniques of Mostow
rigidity that any quasiisometry from $H^n$ to itself extends to a
quasiconformal map from the ideal boundary $S^{n-1}$ to itself.
Also it is clear that if the extended map is a conformal map, then
the original map is isotopic to an isometry.

Here we use a technique of Sullivan \cite{sul}. The action of a
discrete group $\Gamma$ on a measure space is {\sl conservative}
if there is no positive measure set $A$ so that the translates
$\{\gamma(A): \gamma \in \Gamma\}$ are disjoint up to measure zero
set. In our case, we are using a Lebesque measure on $S^{n-1}$.
The following is due to Sullivan.
\begin{thm}A discrete torsion free subgroup of $Iso(H^n)$ admits no invariant
$k$-plane field $(0<k<n-1)$ on the part of $S^{n-1}_\infty$ where
its action is conservative.
\end{thm}
A direct proof can be found in \cite{Mcc}.

First we prove the following observation.
\begin{lemma}\label{conservative}Suppose $\Gamma$ is a discrete group of $Iso(H^n)$.
If the convex core of $\Gamma$ has an upper bound on the
injectivity radius, then $\Gamma$ acts conservatively on its limit
set $L_\Gamma$.
\end{lemma}
\begin{pf}Actually this statement is proved in Theorem 5.11 of
\cite{MT} for $H^3$ and the same proof works for any dimension
$\geq 3$, but for reader's convenience we sketch the proof.
Suppose not. Then there is a positive measure set $A \subset
L_\Gamma$ so that $\{\gamma(A): \gamma \in \Gamma\}$ are disjoint
up to measure zero sets. Then it is not difficult to see that for
a Dirichlet polyhedron $P_a=\{p\in H^n| d(p,a)\leq d(p,\gamma(a))\
\text{for any}\ \gamma\in \Gamma\}$,\ $\overline{P_a}\cap
L_\Gamma$ has positive measure. Then a geodesic ray from $a$ to a
Lebesgue density point in $\overline{P_a}\cap L_\Gamma$, which is
contained in $P_a$ and the convex hull of $L_\Gamma$, has
unbounded distance to $\partial P_a$. This shows that along this
geodesic ray, the injectivity radius tends to infinity since $P_a$
is a fundamental domain of $\Gamma$.
\end{pf}
For dimension 3, due to the solution of Ahlfors conjecture
\cite{agol, gabai}, the limit set of a finitely generated Kleinian
group has either zero measure or whole measure. If it has zero
measure, the action on the limit set is conservative automatically,
if whole measure, the limit set is $S^2$ and the action is ergodic,
so conservative. See \cite{Ca3}. But as far as we know, Ahlfors
conjecture is not known for dimension $\geq 4$.

 From these two facts we will
derive a useful conclusion. For dimension 3, it is a classical
result by Ahlfors-Bers theory. But for higher dimension, the
following theorem is not on the literature.
\begin{thm}\label{endisometric}Let $M=H^n/\Gamma$ be a convex cocompact hyperbolic $n\geq 3$-manifold.
If $N$ is a convex cocompact hyperbolic $n$-manifold homotopy
equivalent to $M$ so that corresponding ends are isometric by
isometries preserving the markings, then it is isometric to $M$.
\end{thm}
\begin{pf}
Let $\phi:M \ra N$ be a proper homotopy equivalence extending an
isometry on ends. Then its lift $\tilde\phi:H^n \ra H^n$ extends to
$S^{n-1}_\infty$ as a quasi-conformal map $\phi':S^{n-1}_\infty \ra
S^{n-1}_\infty$. Then $\phi'$ is differentiable almost everywhere.
The pullback of the spherical metric $\sigma$ by $\phi'$ determines
an ellipsoid in the tangent space to almost every point of
$S^{n-1}_\infty$. The vectors maximizing the ratio
$(\phi')^*\sigma(v)/\sigma(v)$ span a canonical subspace $E_x\subset
T_xS^{n-1}_\infty$, which cuts the ellipsoid in a round sphere of
maximum radius. If $\phi'$ is not conformal there is a positive set
on which $E_x$ has constant rank, which defines a $k$-plane field
invariant under $\Gamma$.

By Lemma \ref{conservative} and Sullivan's theorem, $\Gamma$
cannot have an invariant $k$-plane on the limit set. Now it
suffices to show that $\phi'$ is conformal on the domain of
discontinuity. But it follows from the assumption that the
corresponding ends are isometric.
\end{pf}
The other application is the following. Note here that we are using
a smooth  convex boundary, not the boundary of the convex core.
\begin{thm}\label{boundarymetric}Let $M$ be a $n$-manifold, $n>3$, which
admits a convex cocompact hyperbolic metric of constant sectional
curvature. Suppose $M$ admits a hyperbolic metric with a convex
smooth boundary. Then the induced metric on the boundary determines
the hyperbolic metric uniquely.
\end{thm}
\begin{pf}First we will show that the induced metric determines the
second fundamental form. But it is well-known that a hypersurface in
$H^n$ is completely determined by its induced metric and the second
fundamental form. In conclusion, a hypersurface in $H^n,\ n>3$ is
completely determined by its induced metric.  Once this is proved,
we can prove the theorem as follows. Lift M with a convex smooth
boundary into $H^n$. Then it is a domain $U$ bounded by
hypersurfaces corresponding to boundary components of $M$ in $H^n$.
Since each hypersurface $H$ is completely determined by its induced
metric, a component of $H^n\setminus U$ bounded by $H$ is completely
determined. Its quotient under the stabilizer of $H$ descends to an
end of $M$. So each end of $M$ is completely determined by the
induced metric on the boundary. Then by Theorem \ref{endisometric},
the hyperbolic metric on $M$ is uniquely determined.

Hence it suffices to show that the induced metric determines the
second fundamental form for dimension $>3$.

Let 
$II$ be the real-valued second fundamental form and 
$B$ symmetric operators defined by
$$II(x,y)=I(Bx,y)=I(x,By).$$
Or if $D$ is the Levi-Civita connection on $H^n$ and $S$ is a
hypersurface in it, then the shape operator $B:TS\ra TS$ is
defined by $Bx=-D_x N$ where $N$ is a unit normal vector field to
$S$ and it satisfies $II(x,y)=I(Bx,y)=I(x,By)$.

%

One can diagonalise $B$ with respect to an orthonormal basis $e_i$
with eigenvalues $b_i$. But the eigenvalues of $B$ can be deduced
from the sectional curvatures of the induced metric. By the Gauss
formula
$$K(x,y)=\tilde K(x,y)+ \bar K(x,y),$$ where $\tilde K$ is the
sectional curvature of $H^n$, which is $-1$ in this case, and
$\bar K$ is the Gaussian curvature. If $\{e_i\}_{i=1}^{n-1}$ is an
orthonormal basis for the hypersurface which diagonalize $II$, we
have
$$K(e_i,e_j)=-1 + b_ib_j.$$
If $n-1\geq 3$, $b_i$ has a solution.


Hence for a given induced metric, one can diagonalise
the second fundamental form using an orthogonal basis, and the
eigenvalues can be calculated from the sectional curvatures, so
the second funamental form is determined by the induced metric.
\end{pf}

\vskip .3 in
\noindent Gilles Courtois\\
 L'Institut de Mathématiques de Jussieu\\
4, place Jussieu, 75252 Paris Cedex 09, France\\
e-mail: courtois\char`\@ math.jussieu.fr\\

\noindent Inkang Kim\\
KIAS, School of Mathematics\\
Hoegiro 85, Dongdaemun-gu\\
Seoul, 130-722 Korea\\
e-mail: inkang\char`\@ kias.re.kr

\end{document}